\documentclass[12pt]{amsart}
\usepackage{amsfonts}
\usepackage{amsmath}
\usepackage{amssymb}
\usepackage{amsthm}
 \usepackage{graphicx,latexsym}
 \usepackage[hmargin=1.2in, vmargin=1.2in]{geometry}
\newtheorem{theorem}{Theorem}
\newtheorem{proposition}{Proposition}
\newtheorem{lemma}{Lemma}

\theoremstyle{definition}

\newtheorem{example}{Example}

\theoremstyle{remark}
\newtheorem{remark}{Remark}

\begin{document}
\title[On general prime number theorems with remainder]{On general prime number theorems with remainder}

\author[G. Debruyne]{Gregory Debruyne}
\address{G. Debruyne\\ Department of Mathematics\\ Ghent University\\ Krijgslaan 281\\ B 9000 Gent\\ Belgium}
\thanks{G. Debruyne gratefully acknowledges support by Ghent University, through a BOF Ph.D. grant}
\email{gregory.debruyne@UGent.be}
\author[J. Vindas]{Jasson Vindas} 
\thanks{The work of J. Vindas was supported by the Research Foundation--Flanders, through the FWO-grant number 1520515N}
\address{J. Vindas\\ Department of Mathematics\\ Ghent University\\ Krijgslaan 281\\ B 9000 Gent\\ Belgium}
\email{jasson.vindas@UGent.be}

\subjclass[2010]{11N80, 11M45.}
\keywords{The prime number theorem;  zeta functions; Tauberian theorems for Laplace transforms; Beurling generalized primes; Beurling generalized integers}

\begin{abstract}
We show that for Beurling generalized numbers the prime number theorem in remainder form
$$\pi(x) = \operatorname*{Li}(x) + O\left(\frac{x}{\log^{n}x}\right) \quad \mbox{for all } n\in\mathbb{N}$$
is equivalent to (for some $a>0$)
$$N(x) = ax + O\left(\frac{x}{\log^{n}x}\right) \quad \mbox{for all } n \in \mathbb{N},$$ 
where $N$ and $\pi$  are the counting functions of the generalized integers and primes, respectively.
This was already considered by Nyman (Acta Math. 81 (1949), 299--307), but his article on the subject contains some mistakes. We also obtain an average version of this prime number theorem with remainders in the Ces\`{a}ro sense.
\end{abstract}

\maketitle

\section{Introduction}

Since the prime number theorem (PNT) was proved in 1896, independently by Hadamard and de la Vall\'ee-Poussin, mathematicians have wondered which conditions on the primes (and the integers) were really necessary to prove this kind of theorems. For this reason, Beurling introduced in \cite{beurling} the idea of generalized prime numbers. A real sequence $\{ p_{k}\}_{k \in \mathbb{N}}$ is said to be a (Beurling) generalized prime number system if it merely satisfies
\[
 1 < p_{1} \leq p_{2} \leq \dots \leq p_{k} \rightarrow \infty.
\]

The set of generalized integers \cite{bateman-diamond,beurling} is the semi-group generated by the generalized primes. We arrange the generalized integers in a non-decreasing sequence
\[
 1=n_{0}< n_{1} \leq n_{2} \leq \dots \leq n_{k} \rightarrow \infty,
\]
where one takes multiplicities into account. The central objects here are the counting functions of the generalized primes and integers, denoted as
\begin{equation}
\label{defcounting}
 \pi(x) = \sum_{p_{k}\leq x} 1 \quad \mbox{ and } \quad N(x) = \sum_{n_{k}\leq x} 1.
\end{equation}
A typical question in this setting is to determine conditions on $N$, as mild as possible, such that the PNT still holds. This question for the PNT in the form 
\begin{equation}
\label{eqPNTusual}
\pi(x)\sim \frac{x}{\log x}
\end{equation}
 has been studied quite intensively, starting with the seminal work of Beurling \cite{beurling}. We refer to \cite{bateman-diamond, beurling, kahane1, s-v, zhang2015} for results in this direction. 

In this article we are interested in stronger PNT versions than $(\ref{eqPNTusual})$ for Beurling generalized primes. Our aim is to study the PNT with remainder
\begin{equation}
\label{eqpntremainder}
 \pi(x) = \operatorname*{Li}(x) + O_{n}\left( \frac{x}{\log^{n} x}\right), \ \ \ \text{for all } n \in \mathbb{N}\: ,
\end{equation}
where $\operatorname*{Li}$ stands for the logarithmic integral. Naturally (\ref{eqpntremainder}) is equivalent to the asymptotic expansion
$$
\pi(x)\sim \frac{x}{\log x}\sum_{n=0}^{\infty}\frac{n!}{\log^{n}x}.
$$
The following theorem will be shown:

\begin{theorem} \label{thmain} The PNT with remainder (\ref{eqpntremainder}) holds if and only if the generalized integer counting function $N$ satisfies (for some $a>0$)
\begin{equation} \label{eqiremainderN}
 N(x) = ax + O_n\left(\frac{x}{\log^{n}x} \right), \ \ \ \text{for all } n \in \mathbb{N}.
\end{equation}   

\end{theorem}

Nyman has already stated Theorem \ref{thmain} in \cite{nyman}, but his proof contained some mistakes \cite{inghamreview}. It is not true that his condition \cite[statement~(B), p.~300]{nyman}, in terms of the zeta function 
\begin{equation}
\label{defzeta}
 \zeta(s) = \sum_{k=0}^{\infty} n_{k}^{-s} = \int^{\infty}_{1^{-}} x^{-s} \mathrm{d}N(x),
\end{equation}
is equivalent to either (\ref{eqpntremainder}) or (\ref{eqiremainderN}) (see Examples \ref{example1Nyman}--\ref{example3Nyman} below) and his proof has several gaps.

We will show a slightly more general version of Theorem \ref{thmain} in Section \ref{Section Nyman PNT} which also applies to non-discrete generalized number systems (cf. Section \ref{preli}). For it, we first obtain a complex Tauberian remainder theorem in Section \ref{section Tauberian theorems}, and we then give a precise translation of (\ref{eqpntremainder}) and (\ref{eqiremainderN}) into properties of the zeta function. In Section \ref{Section Cesaro PNT with remainder} we provide a variant of Theorem \ref{thmain} in terms of Ces\`{a}ro-Riesz means of the remainders in the asymptotic formulas  (\ref{eqpntremainder}) and (\ref{eqiremainderN}).

\section{Preliminaries and Notation} \label{preli}

\subsection{Beurling Generalized Number Systems}
We shall consider an even broader definition of generalized numbers \cite{beurling}, which includes the case of non-necessarily discrete number systems. 

A \emph{(Beurling) generalized number system} is merely a pair of non-decreasing right continuous functions $N$ and $\Pi$ with $N(1)=1$ and $\Pi(1)=0$, both having support in $[1,\infty)$, and linked via the relation

\begin{equation}
\label{defzetaextended}
\zeta(s) :=\int^{\infty}_{1^{-}} x^{-s}\mathrm{d}N(x)= \exp\left(\int^{\infty}_{1}x^{-s}\mathrm{d}\Pi(x)\right),
\end{equation}
on some half-plane where the integrals are assumed to be convergent. We refer to $N$ as the generalized number distribution function and call $\Pi$ the Riemann prime distribution function of the generalized number system. These functions uniquely determine one another; in fact, $\mathrm{d}N=\exp^{\ast}(\mathrm{d}\Pi)$, where the exponential is taken with respect to the multiplicative convolution of measures \cite{diamond1}. We are only interested in generalized number systems for which the region of convergence of the zeta function (\ref{defzetaextended}) is at least $\Re e\:s>1$, and hence we assume this condition in the sequel\footnote{This assumption is actually no restriction at all. In fact, if the zeta only converges on $\Re e\: s>\alpha>0$, one may then perform a simple change of variables and replace $N$ and $\Pi$ by the generalized number system $\alpha N(x^{1/\alpha})$ and $\alpha \Pi(x^{1/\alpha})$.}. The latter assumption clearly implies that $N(x)$ and $\Pi(x)$ are both $O(x^{1+\varepsilon})$, for each $\varepsilon>0$. 

If $N$ is the counting function of a discrete number system with prime counting function $\pi$, as defined in the Introduction via (\ref{defcounting}), the Riemann prime counting function of the discrete generalized number system is given by 
\begin{equation}
\label{defriemann}
\Pi(x) = \sum^{\infty}_{j=1} \frac{\pi(x^{1/j})}{j}\: .
\end{equation}
Since $\pi(x)$ vanishes for $x<p_1$, the sum (\ref{defriemann}) is actually finite and in particular convergent. It is not difficult to verify that (\ref{defriemann}) satisfies (\ref{defzetaextended}); indeed, 
$$
\exp\left(\int_{1}^{\infty}x^{-s}\mathrm{d}\Pi(x)\right)=\exp\left(-\int_{1}^{\infty}\log (1-x^{-s})\mathrm{d}\pi(x)\right)=\prod_{k=1}^{\infty}\left(1-p_{k}^{-s}\right)^{-1},
$$
and thus (\ref{defzetaextended}) becomes in this case a restatement of the well-known Euler product formula for the zeta function of a discrete generalized number system \cite{bateman-diamond}. 

The function $\Pi$ may be replaced by $\pi$ in virtually any asymptotic formula about discrete generalized primes. More precisely, we have that $0\leq \Pi(x)-\pi(x)\leq \pi(x^{1/2})+\pi(x^{1/3})\log x/\log p_1$; in particular, $\Pi(x)=\pi(x)+O(x^{1/2+\varepsilon})$, for each $\varepsilon>0$, for a discrete generalized number system satisfying our assumption that its associated zeta function $\zeta(s)$ converges on $\Re e\:s>1$. Naturally, a Chebyshev type bound $\pi(x)=O(x/\log x)$ yields the better asymptotic relation $\Pi(x)=\pi(x)+O(x^{1/2}/\log x)$.
However, we mention that, in general, it is not always possible to determine a non-decreasing function $\pi$ satisfying (\ref{defriemann}) for (non-discrete) generalized number systems as defined above (cf. \cite{hilberdink}). Therefore, we only work with $\Pi$ in order to gain generality.
 
\subsection{Fourier Transforms and Distributions}
Fourier transforms, normalized as $\hat{f}(t)=\int_{-\infty}^{\infty}e^{-itx}f(x)\:\mathrm{d}x$, will be taken in the sense of tempered distributions; see \cite{estrada-kanwal,vladimirov} for distribution theory. The standard Schwartz spaces of compactly supported and rapidly decreasing smooth test functions are denoted as usual by $\mathcal{D}(\mathbb{R})$ and $\mathcal{S}(\mathbb{R})$, while $\mathcal{D}'(\mathbb{R})$ and $\mathcal{S}'(\mathbb{R})$ stand for their topological duals, the spaces of distributions and tempered distributions. The dual pairing between a distribution $f$ and a test function $\phi$ is denoted as $\langle f,\phi\rangle $, or as $\langle f(u),\phi(u) \rangle$ with the use of a variable of evaluation; when $f$ is a regular distribution we of course have  $\langle f(u),\phi(u) \rangle=\int_{-\infty}^{\infty}f(u)\phi(u)\mathrm{d}u$. For $f\in\mathcal{S}'(\mathbb{R})$, its Fourier transform $\hat{f}\in\mathcal{S}'(\mathbb{R})$ is defined via  duality as 
$\langle \hat{f}(t),\phi(t)\rangle:=\langle f(u),\hat{\phi}(u)\rangle$, for each $\phi\in\mathcal{S}(\mathbb{R})$. If $f\in\mathcal{S}'(\mathbb{R})$ has support in $[0,\infty)$, its Laplace transform is
$\mathcal{L}\left\{f;s\right\}:=\left\langle f(u),e^{-su}\right\rangle,$ 
$\Re e\:s>0,$ and its Fourier transform $\hat{f}$ is the distributional boundary value of $\mathcal{L}\left\{f;s\right\}$ on $\Re e\:s=0$. 

We also mention that asymptotic estimates $O(g(x))$ are meant for $x\gg1$ unless otherwise specified.

\section{A Tauberian Theorem with Remainder}\label{section Tauberian theorems}

The following Tauberian remainder theorem for Laplace transforms will be our main tool for translating information on the zeta function of a generalized number system into asymptotic properties for the functions $N$ and $\Pi$ in the next section. Theorem~\ref{thtaub1} extends a Tauberian result by Nyman (cf. \cite[Lemma~II]{nyman}). We point out that our $O$-constants hereafter depend on the parameter $n\in\mathbb{N}$.

\begin{theorem} \label{thtaub1}
Suppose $S$ is non-decreasing and $T$ is a function of (locally) bounded variation such that it is absolutely continuous for large arguments and $T'(x) \leq A e^{x}$ with $A \geq 0$. Let both functions have support in $[0,\infty)$. Assume that
\[ 
 G(s) = \int^{\infty}_{0^{-}} e^{-su} (\mathrm{d}S(u)-\mathrm{d}T(u)) \quad\text{is convergent for } \Re e \: s > 1
\]
and can be extended to a $C^{\infty}$-function on the line $\Re e \: s = 1$, admitting the following bounds:
\begin{equation} \label{assumption2taub} G^{(n)}(1+it)= O( \left|t\right|^{\beta_{n}}) \quad \mbox{for each } n \in \mathbb{N},
\end{equation}                                                                                           
where the $\beta_{n}$ are such that 
\begin{equation}
\label{eqgrowthexponents}
\lim_{n\to\infty}\frac{\beta_{n}}{n} =0.
\end{equation}
Then,  
the ensuing asymptotic formula holds:
\begin{equation} \label{resulttaub} S(x) = T(x) + O\left(\frac{e^{x}}{x^{n}}\right), \quad \mbox{ for all } n \in \mathbb{N}. 
\end{equation}                                                       
\end{theorem}
\begin{proof} Clearly, by enlarging the exponents in (\ref{assumption2taub}) if necessary, we may assume the $\beta_n$ is a non-decreasing sequence of positive numbers. Modifying $T$ on finite intervals does not affect the rest of the hypotheses, so we assume that $T$ is locally absolutely continuous on the whole $[0,\infty)$ and that the upper bound on its derivative holds globally.  Furthermore, we may assume without loss of generality that $T'(x) \geq 0$. Indeed, if necessary we may replace $S$ by $S + T_{-}$ and $T$ by $T_{+}$, where $T(x) = T_{+}(x) - T_{-}(x)$ with $T_{+}$ and $T_{-}$ the distribution functions of the positive and negative parts of $T'$. Since $T(x)=O(e^{x})$, the Laplace-Stieltjes transform of $S$ also converges on $\Re e\:s>1$. Thus,
\begin{align*}
S(x)=\int_{0^{-}}^{x}\mathrm{d}S(u)\leq e^{\sigma x}\int_{0^{-}}^{\infty}e^{-\sigma u}\mathrm{d}S(u)=O_{\sigma}(e^{\sigma x}), \quad \sigma>1.
\end{align*}
Let us define $\Delta(x) = e^{-x}(S(x) - T(x))$ and calculate its Laplace transform,
\begin{align*} 
\mathcal{L}\{\Delta;s\} & = \int^{\infty}_{0} e^{(-s-1)u}(S(u)-T(u))\mathrm{d}u = \frac{1}{1+s} \int^{\infty}_{0^{-}} e^{(-s-1)u}\mathrm{d}(S-T)(u) \\
& = \frac{1}{1+s} \mathcal{L}\{\mathrm{d}S-\mathrm{d}T;s+1\} = \frac{G(s+1)}{s+1}, \quad \Re e\:s>0,
\end{align*}
where we have used that $\Delta (x)=o(e^{\eta x})$ for each $\eta>0$. 
Setting $s = \sigma + it$ and letting $\sigma \rightarrow 0^{+}$ in this expression in the space $\mathcal{S}'(\mathbb{R})$, we obtain that the Fourier transform of $\Delta$ is the smooth function
$$
\hat{\Delta}(t) = \frac{G(1+it)}{1+it}.
$$ Since $\beta_{n}$ is non-decreasing, we obtain the estimates
\begin{equation}
\label{eqextra5}
\hat{\Delta}^{(n)}(t)=  O((1+|t|)^{\beta_{n}- 1}).
\end{equation}

We now derive a useful Tauberian condition on $\Delta$ from the assumptions on $S$ and $T$. If $x \leq y \leq x + \min\{ \Delta(x)/2A,\log(4/3)\}$ and $\Delta(x) > 0$, we find, by using the upper bound on $T'$,
\begin{align*}                                                                         
	\Delta(y) & = \frac{S(y)- T(y)}{e^{y}} \geq \frac{S(x) - T(x)}{e^{x}}\frac{e^{x}}{e^{y}} - A(y-x) \geq \Delta(x) \frac{e^{x}}{e^{y}} - \frac{\Delta(x)}{2}
		\\
		&
		 \geq \frac{\Delta(x)}{4}.
\end{align*} 
Similarly one can show that 
$$
-\Delta(y) \geq -\Delta(x)/2 \quad \mbox{if }  x + \Delta(x)/2A \leq y \leq x \mbox{ and } \Delta(x) < 0.
$$
We now estimate $\Delta(h)$ in the case $\Delta(h) > 0$. Set $\varepsilon = \min\{ \Delta(h)/2A,\log(4/3)\}$ and choose $\phi \in \mathcal{D}(0,1)$ such that $\phi \geq 0$ and $\int_{-\infty}^{\infty} \phi(x)\mathrm{d}x = 1$. We obtain
\begin{align*}
 \Delta(h) & = \frac{1}{\varepsilon}\int^{\varepsilon}_{0} \Delta(h) \phi\left(\frac{x}{\varepsilon}\right)\mathrm{d}x 
 \\
 &
 \leq \frac{4}{\varepsilon} \int^{\varepsilon}_{0} \Delta(x+h) \phi \left(\frac{x}{\varepsilon}\right)\mathrm{d}x 
= \frac{2}{\pi} \int^{\infty}_{-\infty} \hat{\Delta}(t) e^{iht} \hat{\phi}(-\varepsilon t) \mathrm{d}t 
 \\
&
= \frac{2}{(ih)^{n}\pi} \int^{\infty}_{-\infty} e^{iht} \left(\hat{\Delta}(t) \hat{\phi}(-\varepsilon t)\right)^{(n)} \mathrm{d}t
\\
&
 \leq \frac{2}{h^{n}\pi}\sum^{n}_{j=0} {n \choose j} \int^{\infty}_{-\infty} \left|\hat{\Delta}^{(j)}\left(\frac{t}{\varepsilon}\right) \hat{\phi}^{(n-j)}(-t)\right| \varepsilon^{n-j-1} \mathrm{d}t
 \\
 &
 = O\left(\frac{1}{h^{n} \varepsilon^{\beta_{n}}}\right),
\end{align*}
where we have used $\hat{\phi} \in \mathcal{S}(\mathbb{R})$ and (\ref{eqextra5}). If $\Delta(x) < 0$ one gets an analogous estimate by using a $\phi \in \mathcal{D}(-1,0)$ such that $\phi \geq 0$ and $\int_{-\infty}^{\infty} \phi(x)\mathrm{d}x =  1$. If $\varepsilon = \log(4/3)$, it clearly follows that $\Delta(h) = o(1)$ and we may thus assume that $\varepsilon = \Delta(h)/2A$. This gives that $\Delta(h) = O_{n}(h^{-n/(\beta_{n} + 1)})$ which proves (\ref{resulttaub}) because of (\ref{eqgrowthexponents}).
\end{proof}

We will also need a converse result, an Abelian counterpart. It is noteworthy that the bounds for $G^{(n)}(1+it)$ we get from the converse result are actually much better than the ones needed for Theorem \ref{thtaub1}. 

\begin{proposition} \label{conversetauberian} Let $S$ be a non-decreasing function, let $T$ be of (locally) bounded variation such that it is absolutely continuous for large arguments and $T'(x) \leq Ae^{x}$ for some positive $A$, and let both functions have support in $[0,\infty)$. Suppose that the asymptotic estimate $(\ref{resulttaub})$ holds for all $n$.
Then,
\begin{equation*}
 G(s) = \int^{\infty}_{0^{-}} e^{-su} (\mathrm{d}S(u)-\mathrm{d}T(u)) \ \ \ \text{is convergent for } \Re e \: s \geq 1.
\end{equation*}
Furthermore, $G$ is $C^{\infty}$ on $\Re e \: s = 1$ and for each $\varepsilon>0$ and $n \in \mathbb{N}$ its $n$-th derivative satisfies the bound 
\begin{equation}
\label{eqboundderivatives}
G^{(n)}(\sigma+it) = O((1+\left|t\right|)^{\varepsilon}), \quad \sigma\geq1,\ t\in\mathbb{R},
\end{equation} 
with global $O_{\varepsilon,n}$-constants.
Moreover, if $T'(x) \leq B x^{-1}e^{x}$ for some positive $B$ and $x\gg 1$, then the better asymptotic estimate 
\begin{equation}
\label{eq bound G}
G(\sigma+it)=o(\log|t|)
\end{equation}
is valid uniformly for $\sigma\geq1$ as  $|t|\to\infty$.
\end{proposition}
\begin{proof}
As in the proof of Theorem \ref{thtaub1}, we may assume that $T$ is locally absolutely continuous on $[0,\infty)$ and $0\leq T'(x) \leq Ae^{x}$. From the assumptions it is clear that $S$ as well as $T$ are $O(e^{x})$. The asymptotic estimates (\ref{resulttaub}) obviously give the convergence of $G(s)$ for $\Re e\:s\geq 1$ and the fact that $G$ is $C^{\infty}$ on $\Re e\: s=1$. Let us now show the asymptotic bounds (\ref{eqboundderivatives}). It is clear that it holds with $\varepsilon=0$ on the half-plane $\sigma\geq2$. We thus restrict our attention to the strip $1\leq \sigma<2$. We keep $|t|\geq 1$.  Let $X \gg 1$ be a constant, which we will specify later. We have 
\begin{align} 
G(s) =  &\int^{X}_{0^{-}}e^{-sx} \mathrm{d}S(x)- \left(\int^{X}_{0}e^{-sx}T'(x) \mathrm{d}x +T(0)\right)
\nonumber
\\
&
\label{eqformulaG}
\quad + s \int^{\infty}_{X} e^{-sx} \left(S(x) - T(x)\right) \mathrm{d}x + e^{-sX}\left(S(X)-T(X)\right).
\end{align}
We differentiate the above formula $n$ times and bound each term separately. The first term can be estimated by
\begin{align*}               
\left|\int^{X}_{0^{-}}e^{-sx} (-x)^{n}\mathrm{d}S(x)\right| 
& \leq \int^{X}_{0^{-}}e^{-x}x^{n} \mathrm{d}S(x)
\\
& = e^{-X}X^{n}S(X) + \int^{X}_{0} e^{-x}x^{n}S(x) \mathrm{d}x - n\int^{X}_{0}e^{-x}x^{n-1}S(x) \mathrm{d}x 
\\
& \leq CX^{n+1},
\end{align*}
as $S$ is non-decreasing and $O(e^{x})$. The second term from (\ref{eqformulaG}) can be bounded in a similar way by this quantity, while the last term is even $O(1)$. It thus remains to bound the third term from (\ref{eqformulaG}). Suppose that $S(x) - T(x) = O(e^{x}x^{-\gamma})$, where $\gamma > n + 1$, then
\begin{equation*}
 \left|\int^{\infty}_{X} e^{-sx}x^{n} \left(S(x) - T(x)\right)\mathrm{d}x\right| \leq \int^{\infty}_{X}x^{n-\gamma} \mathrm{d}x \leq C'X^{n-\gamma+1}.
\end{equation*} 
Combining these inequalities and choosing $X = \left|t\right|^{1/\gamma}$, we obtain
\begin{equation*}
 \left|G^{(n)}(\sigma+it)\right| \leq C'' X^{n+1} + C' (2+\left|t\right|)X^{n-\gamma +1} = O\left(\left|t\right|^{\frac{n+1}{\gamma}}\right).
\end{equation*}
Since $\gamma$ can be chosen arbitrarily large, (\ref{eqboundderivatives}) follows. 

The proof of (\ref{eq bound G}) is similar if we work under the assumption $T'(x) \leq Bx^{-1}e^{x}$. This bound implies that $T(x) \ll \operatorname*{Li}(e^{x}) = O(x^{-1}e^{x})$, which gives $S(x)=O(x^{-1}e^{x})$ as well. The starting point is again the formula (\ref{eqformulaG}) for $G$. Via the same reasoning as above, the bounds for the first and second term, in case $n=0$, can be improved to $O(\log X)$. Employing the same bound for the third term, we obtain the result after choosing $X = \left|t\right|^{1/(\gamma -1)}$ and letting $\gamma \rightarrow \infty$.

\end{proof}                       

\section{The PNT with Nyman's Remainder}\label{Section Nyman PNT}

We establish in this section the following general form of Theorem \ref{thmain}:

\begin{theorem} \label{thPNTNyman}
 For a generalized number system, the following four statements are equivalent:
\begin{itemize}
\item[$(i)$] For some $a > 0$, the generalized integer distribution function $N$ satisfies
\begin{equation} 
\label{asymNNremainder}
 N(x) = ax + O\left(\frac{x}{\log^{n}x} \right), \quad \text{for all } n \in \mathbb{N}.
\end{equation}
\item[$(ii)$]  For some $a>0$, the function 
\begin{equation}
\label{eqzetapole}
G(s)= \zeta(s)-\frac{a}{s-1}
\end{equation}
has a $C^{\infty}$-extension to $\Re e\:s\geq 1$ and there is some $\varepsilon>0$ such that
\begin{equation}
\label{eqzetapolederivativesbound}
G^{(n)}(1+it)=O(|t|^{\varepsilon}), \quad \text{for all } n \in \mathbb{N}.
\end{equation}             
\item[$(iii)$] For some $a>0$ and each $\varepsilon>0$, the function (\ref{eqzetapole}) satisfies
\begin{equation}
\label{eq2zetapolederivativesbound}
G^{(n)}(\sigma+it)=O((1+|t|)^{\varepsilon}), \quad \sigma>1,\: t\in\mathbb{R}, \quad \text{for all } n \in \mathbb{N},
\end{equation}        
with global $O_{\varepsilon,n}$-constants.                                                       
\item [$(iv)$] The  Riemann prime distribution function $\Pi$ satisfies
\begin{equation}
\label{relationPi} 
 \Pi(x) = \operatorname*{Li}(x) + O\left(\frac{x}{\log^{n}x} \right), \quad \text{for all } n \in \mathbb{N}.
\end{equation}
\end{itemize} 
\end{theorem}
\begin{remark}\label{rk1Nyman} The condition $(iii)$ implies the apparently stronger assertion that $G$ has a $C^{\infty}$-extension to $\Re e\:s\geq1$ and that (\ref{eq2zetapolederivativesbound}) remains valid for $\sigma\geq1$, as follows from a standard local $L^{\infty}$ weak$^{\ast}$ compactness argument.
\end{remark}

Before giving a proof of Theorem \ref{thPNTNyman}, we make a comment on reference \cite{nyman}. Therein, Nyman stated that the conditions $(i)$ and $(iv)$ from Theorem \ref{thPNTNyman} were also equivalent to: for each $\varepsilon>0$ and $n\in\mathbb{N}$ 
\begin{equation}
\label{eqNymanwrong}
\zeta^{(n)}(\sigma+it)=O(|t|^{\varepsilon}) \quad \mbox{ and } \quad \frac{1}{\zeta(\sigma+it)}=O(|t|^{\varepsilon}),
\end{equation}
uniformly on the region $\sigma>1$ and $|t|\geq \varepsilon$. It was noticed by Ingham in Mathematical Reviews \cite{inghamreview} that (\ref{eqNymanwrong}) fails to be equivalent to (\ref{asymNNremainder}) and (\ref{relationPi}). In fact (\ref{eqNymanwrong}) can hardly be equivalent to any of these two asymptotic formulas because it does not involve any information about $\zeta$ near $s=1$, contrary to our conditions $(ii)$ and $(iii)$. A large number of counterexamples to Nyman's statement can easily be found among zeta functions arising as generating functions from analytic combinatorics and classical number theory.  We discuss three examples here, the first of them is due to Ingham \cite{inghamreview}, while the second one was suggested by W.-B. Zhang.

\begin{example}
\label{example1Nyman} Consider the generalized primes given by $p_k=2^{k}$. The prime counting function for these generalized primes clearly satisfies $\pi(x)=\log x/\log 2+O(1)$ and therefore (\ref{relationPi}) does not hold. The bound $\pi(x)=O(\log x)$ gives that its associated zeta function is analytic on $\Re e\:s>0$ and satisfies $\zeta^{(n)}(s)=O(1)$ uniformly on any half-plane $\Re e\:s\geq \sigma_{0}>0$. We also have the same bound for $1/\zeta(s)$ because $|\zeta(\sigma)||\zeta(\sigma+it)|\geq1$, which follows from the trivial inequality $1+\cos \theta\geq 0$ (see the 3-4-1 inequality in the proof of Lemma \ref{leminversezeta} below). In particular, Nyman's condition (\ref{eqNymanwrong}) is fulfilled. The generalized integer counting function $N$ does not satisfy (\ref{asymNNremainder}), because, otherwise, $\zeta$ would have a simple pole at $s=1$. Interestingly, in this example $N(x)=\sum_{2^{k}\leq x}p(k)$, where $p$ is the unrestricted partition function, which, according to the celebrated Hardy-Ramanujan-Uspensky formula, has asymptotics
\begin{equation}
\label{Matulaeq5}
p(n)\sim \frac{e^{\pi \sqrt{\frac{2n}{3}}}}{4n\sqrt{3}}\: .
\end{equation}  
From (\ref{Matulaeq5}) one easily deduces
\begin{equation}
\label{Matulaeq6}
N(x)\sim A\: \frac{e^{\pi \sqrt{\frac{2\log x}{3\log 2}}}}{\sqrt{\log x}}\: ,
\end{equation}
with $A= (2\pi\sqrt{2})^{-1}\sqrt{\log 2}$,
but (\ref{Matulaeq5}) and (\ref{Matulaeq6}) simultaneously follow from Ingham's theorem for abstract partitions \cite{Ingham1941}.

\end{example}

\begin{example}\label{example2Nyman} A simple example is provided by the generalized prime number system 
$$
2,2,3,3,5,5,\dots,p,p,\dots,
$$ 
that is, the generalized primes consisting of ordinary rational primes $p$ each taken exactly twice. The set of generalized integers for this example then consists of ordinary rational integers $n$, each repeated $d(n)$ times, where $d(n)$ is the classical divisor function.  In this case the associated zeta function to this number system is the square of the Riemann zeta function, which clearly satisfies Nyman's condition (\ref{eqNymanwrong}). On the other hand, Dirichlet's well known asymptotic estimate for the divisor summatory function and the classical PNT yield
$$
N(x)=\sum_{n\leq x}d(n)= x\log x+ (2\gamma-1)x+O(\sqrt{x})
$$
and
$$
\Pi(x)=2\operatorname*{Li}(x)+O(x\exp(-c\sqrt{\log x})).
$$
\end{example}

\begin{example}\label{example3Nyman} This example and Example \ref{example1Nyman} are of similar nature. This time we use generalized integers that arise as coding numbers of certain (non-planar) rooted trees via prime factorization \cite{Matula1968}. Consider the set of generalized primes given by the subsequence $\{p_{2^{k}}\}_{k=0}^{\infty}$ of ordinary rational primes, where $\{p_{k}\}_{k=1}^{\infty}$ are all rational primes enumerated in increasing order. Using the classical PNT for rational primes, one verifies that the prime counting function $\pi$ of these generalized primes satisfies 
$$
\pi(x)=\frac{\log x}{\log 2}- \frac{\log \log x}{\log 2}+O(1).
$$
By the same reason as above, one obtains that the zeta function of these generalized numbers satisfies Nyman's condition (\ref{eqNymanwrong}). The generalized integers corresponding to this example are actually the Matula numbers of rooted trees of height $\leq$2, whose asymptotic distribution was studied in \cite{v-v-w}; its generalized integer counting function $N$ satisfies
$$
N(x)\sim A(\log x)^{\frac{\log\left(\pi /\sqrt{6\log 2}\right)}{2\log 2}}
\exp\left(\pi\sqrt{\frac{2\log x}{3\log 2}} - \frac {(\log\log x)^{2}}{8\log 2}\right),
$$
for a certain constant $A>0$, see \cite[Thm.~1]{v-v-w}. 
\end{example}
\smallskip

The rest of this section is dedicated to the proof of Theorem \ref{thPNTNyman}. First we derive some bounds on the inverse of the zeta function and the non-vanishing of $\zeta$ on $\Re e\:s=1$.

\begin{lemma} \label{leminversezeta} Suppose that condition $(iii)$ from Theorem \ref{thPNTNyman} holds. Then, $(s-1)\zeta(s)$ has no zeros on $\Re e\:s\geq 1$ and, in particular, $1/\zeta(s)$ has a $C^{\infty}$-extension to $\Re e\:s\geq 1$ as well. Furthermore, for each $\varepsilon > 0$, 
\begin{equation}
\label{eqinversezeta} \frac{1}{\zeta(\sigma+it)} = O\left((1+\left|t\right|)^{\varepsilon}\right), \quad\sigma\geq 1,\ t\in\mathbb{R},
\end{equation}
with a global $O_{\varepsilon}$-constant.
\end{lemma}

\begin{proof} 
We use $(iii)$ in the form stated in Remark \ref{rk1Nyman}. The non-vanishing property of $\zeta$ follows already from results of Beurling \cite{beurling}, but, since we partly need the argument in the process of showing (\ref{eqinversezeta}), we also prove this fact for the sake of completeness. Let $t\neq0$.
 We closely follow Hadamard's classical argument \cite{ingham} based on the elementary 3-4-1 trigonometric inequality, that is,
\[ P(\theta):=3 + 4 \cos(\theta) +  \cos(2\theta) \geq 0.
\]
Using the expression $\zeta(s) = \exp\left(\int_{1^{-}}^{\infty}x^{-s}\mathrm{d}\Pi(x)\right)$ and the 3-4-1 inequality, one derives, for $1< \eta $,
\[3\log|\zeta(\eta)|+4\log |\zeta(\eta+it)|+\log |\zeta(\eta+2it)|
= \int_{1^{-}}^{\infty}x^{-\eta}P(t\log x)\mathrm{d}\Pi(x)\geq 0,
\]
namely,
\begin{equation*}
 \left|\zeta^{3}(\eta)  \zeta^{4}(\eta + it )  \zeta(\eta + 2it)\right| \geq 1.
\end{equation*}
This 3-4-1 inequality for $\zeta$ already implies that $1/\zeta(\eta+it)=O(1)$ uniformly on $\eta\geq2$. We assume in the sequel that $1<\eta<2$. Since $\zeta(\eta) \sim a/(\eta-1)$ as $\eta\to1^{+}$, we get 
\begin{align}
 (\eta - 1)^{3} &\leq (\eta - 1)^{3}\zeta^{3}(\eta) \left|\zeta^{4}(\eta + it)\right|\left|\zeta(\eta + 2it)\right| 
\nonumber \\
&\leq A \left|\zeta(\eta + it)\right|^{4} \left|t\right|^{\varepsilon}.
\label{zeta lower}
\end{align}  
As is well known, (\ref{zeta lower}) yields that $\zeta(1+it)$ does not vanish for $t\neq0$. Indeed, if $\zeta(1+it_0)=0$, the fact that $\zeta(s)$ and $\zeta'(s)$ have continuous extensions to $\Re e\:s=1$ would imply $(\eta-1)^{3}= O(|\zeta(\eta+it_0)|^{4})=O((\int_{1}^{\eta}|\zeta'(\lambda+it_{0})|\mathrm{d}\lambda)^{4})=O((\eta-1)^{4})$, a contradiction. The assertions about the $C^{\infty}$-extensions of $(s-1)\zeta(s)$ and $1/\zeta(s)$ must be clear, in particular $1/\zeta(1)=0$. 

Let us now establish the bound (\ref{eqinversezeta}) on the range $1\leq \sigma\leq 2$. We keep here $|t|\gg 1$.      
If $1 \leq \sigma \leq \eta<2$, we find
\[ \left|\zeta(\sigma + it ) - \zeta(\eta + it)\right| = \left|\int^{\eta}_{\sigma} \zeta'(u+it)du\right| \leq A' (\eta - 1) \left|t\right|^{\varepsilon},
\]
where we have used the bound (\ref{eqzetapolederivativesbound}) for $\zeta'$. Combining this inequality with (\ref{zeta lower}), we find
\begin{align*}
\left|\zeta(\sigma + it)\right| &\geq \left|\zeta(\eta + it)\right| - A'(\eta - 1)\left|t\right|^{\varepsilon}\\
&\geq \frac{(\eta - 1)^{3/4}}{A^{1/4}\left|t\right|^{\varepsilon/4}} - A'(\eta - 1)\left|t\right|^{\varepsilon}.
\end{align*}
Now choose $\eta = \eta(t)$ in such a way that
\[ \frac{(\eta - 1)^{3/4}}{A^{1/4}\left|t\right|^{\varepsilon/4}} = 2A'(\eta - 1)\left|t\right|^{\varepsilon},
\]
i.e.,
\[ \eta = 1 + \frac{1}{A(2A')^{4}|t|^{5\varepsilon}}=1+\frac{A''}{|t|^{5\varepsilon}},
\]
assuming $t$ large enough to ensure $\eta < 2$. Then, in this range,
\[ \left|\zeta(\sigma + it)\right| \geq A'(\eta - 1)\left|t\right|^{\varepsilon} = A'A''\left|t\right|^{-4\varepsilon}.
\]
For the range $1+A''|t|^{-5\varepsilon}\leq \sigma\leq 2$, the estimate (\ref{zeta lower}) with $\sigma$ instead of $\eta$ yields exactly the same lower bound.
\end{proof}
We now aboard the proof of Theorem \ref{thPNTNyman}.

\begin{proof}[Proof of Theorem \ref{thPNTNyman}] Upon setting $S(x)=N(e^{x})$ and $T(x)=ae^{x}$, so that
$$
G(s)=\mathcal{L}\{dS-dT;s\}=\zeta(s)-\frac{a}{s-1},
$$
 Theorem \ref{thtaub1} gives the implication $(ii)\Rightarrow  (i)$, Proposition \ref{conversetauberian} yields  
$(i)\Rightarrow (iii)$, whereas $(iii)\Rightarrow  (ii)$ follows from Remark \ref{rk1Nyman}. So, the first three conditions are equivalent and it remains to establish the equivalence between any of these statements and $(iv)$.

 $(iii)\Rightarrow  (iv)$. We now set $S_1(x) := \Pi(e^{x})$ and 
\begin{equation*}
T_1(x) := \int^{e^{x}}_{1} \frac{1- \frac{1}{y}}{\log y}\mathrm{d}y = \operatorname*{Li} (e^{x}) - \log x + A, \quad x\geq 0.
\end{equation*}
A quick calculation gives an explicit expression for $G_1(s) := \mathcal{L}\{dS_1-dT_1;s\}$, namely,
\begin{equation}
\label{defauxfunct}
G_1(s) = \log \zeta(s) - \log s + \log (s-1)= \log ((s-1)\zeta(s))- \log s,
\end{equation}
with the principal branch of the logarithm.
By Remark \ref{rk1Nyman}, Lemma \ref{leminversezeta}, and the Leibniz rule, we obtain that $G_1(1+it)\in C^{\infty}(\mathbb{R})$ and bounds $G_1^{(n)}(1+it) = O_{\varepsilon,n}(\left|t\right|^{\varepsilon})$, $|t|\gg1$. Another application of Theorem \ref{thtaub1} yields (\ref{relationPi}). 

$(iv)\Rightarrow  (ii)$.
Conversely, let  (\ref{relationPi}) hold and retain the notation $S_{1}$, $T_{1}$, and $G_1$ as above. We apply Proposition \ref{conversetauberian} to $S_1$ and $T_1$ to get that (\ref{defauxfunct}) admits a $C^{\infty}$-extension to $\Re e \: s = 1$ and all of its derivatives on that line are bounded by $O(\left|t\right|^{\varepsilon})$ for each $\varepsilon>0$. This already yields that the function $G(s)$ given by (\ref{eqzetapole}), no matter the value of the constant $a$, has also a $C^{\infty}$-extension to $\Re e \: s = 1$ except possibly at $s = 1$. Moreover, since $T_1(x) = O(e^{x}/x)$, we even get from Proposition \ref{conversetauberian} that  $G_{1}(t)=o( \log \left|t\right|)$ for $|t|\gg 1$, or, which amounts to the same, $\zeta(1+it) = O(\left|t\right|^{\varepsilon})$, for each $\varepsilon>0$. Thus, by this bound and the bounds on the derivatives of $\log \zeta(1+it)$, we have that $\zeta^{(n)}(1+it) = O(\left|t\right|^{\varepsilon})$, as can easily be deduced by induction with the aid of the Leibniz formula. 

Summarizing, we only need to show that there exists $a > 0$ for which $\zeta(s) - a/(s-1)$ has a $C^{\infty}$-extension on the whole line $\Re e \: s = 1$. 
The function $\log ((s-1)\zeta(s))$ however admits a $C^{\infty}$-extension to this line, and its value at $s=1$  coincides with that of the function $G_{1}$, as shown by the expression (\ref{defauxfunct}). Therefore, $(s-1)\zeta(s)$ also extends to $\Re e \: s \geq 1$ as a $C^{\infty}$-function, and its value at $s=1$ can be calculated as $a=\lim_{\sigma\to1^{+}}e^{G_{1}(\sigma)}=e^{G_{1}(1)}>0$, because $G_1(\sigma)$ is real-valued when $\sigma$ is real. Hence $\zeta(s) - a/(s-1)$ has also a $C^{\infty}$-extension to $\Re e\:s\geq 1$ (This follows from the general fact that $t^{-1}(f(t) - f(0))$ is $C^{k-1}$ for a $C^{k}$-function $f$.) This concludes the proof of the theorem. 
\end{proof}

\section{A Ces\`{a}ro Version of the PNT with Remainder} \label{Section Cesaro PNT with remainder}
In this last section we obtain an average version of Theorem \ref{thPNTNyman} where the remainders in (\ref{asymNNremainder}) and (\ref{relationPi}) are taken in the Ces\`{a}ro sense. The motivation of this new PNT comes from a natural replacement of $(ii)$, or equivalently $(iii)$, by a certain weaker growth requirement on $\zeta$. 

Let us introduce some function and distribution spaces. The space $\mathcal{O}_{C}(\mathbb{R})$ consists of all $g\in C^{\infty}(\mathbb{R})$ such that there is some $\beta\in\mathbb{R}$ with $g^{(n)}(t)=O_{n}(|t|^{\beta})$, for each $n\in\mathbb{N}$. This space is well-known in distribution theory. When topologized in a canonical way, its dual space $\mathcal{O}_{C}' (\mathbb{R})$ corresponds to the space of convolutors of the tempered distributions \cite{estrada-kanwal,p-s-v}. Another well known space is that of multipliers of $\mathcal{S}'(\mathbb{R})$, denoted as $\mathcal{O}_{M}(\mathbb{R})$ and consisting of all $g\in C^{\infty}(\mathbb{R})$ such that for each $n\in\mathbb{R}$ there is $\beta_n\in\mathbb{R}$ such that $g^{(n)}(t)=O_{n}(|t|^{\beta_n})$. Of course, we have the inclusion relation $\mathcal{O}_{C}(\mathbb{R})\subsetneq \mathcal{O}_{M}(\mathbb{R})$.

Observe that condition $(ii)$ from Theorem \ref{thPNTNyman} precisely tells that for some $a>0$ the analytic function $G(s)=\zeta(s)-a/(s-1)$ has boundary values on $\Re e\: s=1$ in the space $\mathcal{O}_{C}(\mathbb{R})$, that is, $G(1+it)\in\mathcal{O}_{C}(\mathbb{R})$. We now weaken this membership relation to $G(1+it)\in\mathcal{O}_{M}(\mathbb{R})$. To investigate the connection between the latter condition and the asymptotic behavior of $N$ and $\Pi$, we need to use asymptotics in the Ces\`{a}ro sense. For a locally integrable function $E$, with support in $[0,\infty)$, and $\alpha\in\mathbb{R}$, we write
\begin{equation}
\label{eqcesaro}
E(x)=O\left(\frac{x}{\log^{\alpha} x}\right) \quad (\mathrm{C}) \quad (x\to\infty)
\end{equation}
if there is some (possibly large) $m\in\mathbb{N}$ such that the following average growth estimate holds: 
\begin{equation}
\label{ibpneq6}
\int_{0}^{x} \frac{E(u)}{u}\left(1-\frac{u}{x}\right)^{m}\mathrm{d}u=O\left(\frac{x}{\log^{\alpha}x}\right).
\end{equation}
The order $m$ of the Ces\`{a}ro-Riesz mean to be taken in (\ref{ibpneq6}) is totally irrelevant for our arguments below and we therefore choose to omit it from the notation in (\ref{eqcesaro}).
The meaning of an expression $f(x)=g(x)+O\left(x/\log^{\alpha} x\right)$ in the Ces\`{a}ro sense should be clear. We remark that Ces\`{a}ro asymptotics can also be defined for distributions, see \cite{estrada-kanwal,p-s-v}. The notion of Ces\`{a}ro summability of integrals is well-known, see e.g. \cite{estrada-kanwal}.

We have the following PNT with remainder in the Ces\`{a}ro sense:
\begin{theorem} \label{thPNTRCesaro}
 For a generalized number system the following four statements are equivalent:
\begin{itemize}
\item[$(i)$] For some $a > 0$, the generalized integer distribution function $N$ satisfies
\begin{equation} 
\label{asymNNremainder2}
 N(x) = ax + O\left(\frac{x}{\log^{n}x} \right) \quad (\mathrm{C}), \quad \text{for all } n \in \mathbb{N}.
\end{equation}
\item[$(ii)$] For some $a>0$, the function 
\begin{equation}
\label{eqzetapole2}
G(s)= \zeta(s)-\frac{a}{s-1}
\end{equation}
has a $C^{\infty}$-extension to $\Re e\:s\geq 1$ and $G(1+it)\in\mathcal{O}_{M}(\mathbb{R})$.
\item[$(iii)$]  For some $a>0$, there is a positive sequence $\{\beta_n\}_{n=0}^{\infty}$ such that the function (\ref{eqzetapole2}) satisfies
\begin{equation}
\label{eq2zetapolederivativesbound2}
G^{(n)}(s)=O((1+|s|)^{\beta_n}),\quad \text{for all } n \in \mathbb{N},
\end{equation}        
on $\Re e\:s>1$ with global $O_{n}$-constants.                                                       
\item [$(iv)$] The  Riemann prime distribution function $\Pi$ satisfies
\begin{equation}
\label{relationPi2} 
 \Pi(x) = \operatorname*{Li}(x) + O\left(\frac{x}{\log^{n}x} \right) \quad (\mathrm{C}), \quad \text{for all } n \in \mathbb{N}.
\end{equation}
\end{itemize} 
\end{theorem}

We indicate that, as in Remark \ref{rk1Nyman}, the bounds (\ref{eq2zetapolederivativesbound2}) also imply that $G$ has a $C^{\infty}$-extension to $\Re e\:s\geq 1$ and that (\ref{eq2zetapolederivativesbound2}) remains valid on $\Re e\:s\geq 1$. Naturally, the PNT (\ref{relationPi2}) delivered by Theorem \ref{thPNTRCesaro} is much weaker than (\ref{relationPi}).  Before discussing the proof of Theorem \ref{thPNTRCesaro}, we give a family of examples of generalized number systems which satisfy condition $(ii)$ from Theorem \ref{thPNTRCesaro} but not those from Theorem \ref{thPNTNyman}.
\begin{example}
\label{exNyman4} The family of continuous generalized number systems whose Riemann prime distribution functions are given by
\begin{equation*}
 \Pi_{\alpha}(x) = \int^{x}_{1} \frac{1-\cos(\log^{\alpha} u)}{\log u } \mathrm{d}u\:, \quad \mbox{for }\alpha>1, 
\end{equation*}
was studied in \cite{d-s-v}. It follows from \cite[Thm.~3.1]{d-s-v} that there are constants $a_{\alpha}$ such that their zeta functions have the property that $G_{\alpha}(s)=\zeta_{\alpha}(s)-a_{\alpha}/(s-1)$ are entire. In this case, \cite[Thm.~3.1]{d-s-v} also implies that $G_{\alpha}(1+it)\in\mathcal{O}_{M}(\mathbb{R})$, but it does not belong to $\mathcal{O}_{C}(\mathbb{R})$. 
\end{example}

We need some auxiliary results in order to establish Theorem \ref{thPNTRCesaro}. The next theorem is of Tauberian character. Part of its proof is essentially the same as that of \cite[Lemma~2.1]{d-s-v}, but we include it for the sake of completeness.

\begin{theorem} \label{lemsufcondcesaro} Let $E$ be of locally bounded variation  with support on $[1,\infty)$ and suppose that $E(x)=O(x)\ $ $(\mathrm{C})$. Set
$$
F(s)=\int_{1^{-}}^{\infty}x^{-s}\mathrm{d}E(x) \quad (\mathrm{C}), \quad \Re e\:s>1.
$$
Then, $E$ satisfies (\ref{eqcesaro}) for every $\alpha>0$ if and only if $F$ has a $C^{\infty}$-extension to $\Re e\:s\geq1$ that satisfies $F(1+it)\in \mathcal{O}_{M}(\mathbb{R}_{t})$. If this is the case, then there is a sequence $\{\beta_n\}_{n=0}^{\infty}$ such that for each $n$
\begin{equation}
\label{eqextra1}
F^{(n)}(s)=O((1+|s|)^{\beta_n}),\quad \mbox{on } \Re e\:s\geq1.
\end{equation}  
Furthermore, assume additionally that
\begin{equation}
\label{variationmeasurebd}
V(E,[1,x])=\int_{1^{-}}^{x} |\mathrm{d}E|(u)=O\left(\frac{x}{\log x}\right),
\end{equation}
where $|\mathrm{d}E|$ stands for the total variation measure of $\mathrm{d}E$. Then, 
\begin{equation}
\label{eqextra2}
F^{}(s)=O(\log(1+|\Im m\:s|)),\quad \mbox{on } \Re e\:s\geq1.
\end{equation}  
\end{theorem}
\begin{proof} Note that the Ces\`{a}ro growth assumption implies that $F(s)$ is Ces\`{a}ro summable for $\Re e\:s>1$ and therefore analytic there. Let $F_{1}(s)=F(s)/s$ and $R(u)=e^{-u}E(e^{u})$. It is clear that $F_{1}(s)$ has a $C^{\infty}$-extension to $\Re e\:s=1$ that satisfies $F(1+it)\in \mathcal{O}_{M}(\mathbb{R})$ if and only if $F_1$ has the same property. The latter property holds if and only if $R\in\mathcal{O}'_{C}(\mathbb{R})$. Indeed, since $R\in\mathcal{S}'(\mathbb{R})$ and $F_1(s+1)=\mathcal{L}\{R;s\}$, we obtain that $\hat{R}(t)=F_1(1+it)$, whence our claim follows because the spaces $\mathcal{O}'_{C}(\mathbb{R})$ and $\mathcal{O}_{M}(\mathbb{R})$ are in one-to-one correspondence via the Fourier transform \cite{estrada-kanwal}. 

Now, by definition of the convolutor space, $R\in\mathcal{O}'_{C}(\mathbb{R})$ if and only if $\int_{-\infty}^{\infty} R(u+h)\phi(u)\mathrm{d} u = O(h^{-\alpha})$, for each $\alpha>0$ and $\phi\in\mathcal{D}(\mathbb{R})$ \cite{p-s-v}. Writing $h = \log \lambda$ and $\phi(x) = e^{x}\varphi(e^{x})$, we obtain that $R\in\mathcal{O}'_{C}(\mathbb{R})$ if and only if $E(x)/x$ has the quasiasymptotic behavior \cite{estrada-kanwal,p-s-v}
 \begin{equation} 
 \label{eqquasi1} \frac{E( \lambda x)}{\lambda x} = O\left(\frac{1}{\log^{\alpha}\lambda}\right), \quad \lambda \rightarrow \infty\:, \text{ in } \mathcal{D}(0,\infty)\: ,
\end{equation}
which explicitly means that 
\[ \int^{\infty}_{1} \frac{E( \lambda x)}{\lambda x}  \varphi(x)\mathrm{d}x = O\left(\frac{1}{\log^{\alpha}\lambda}\right), \quad \lambda \rightarrow \infty,
\]
for every test function $\varphi \in \mathcal{D}(0,\infty)$. Using \cite[Thm.~2.37, p.~154]{p-s-v}, we obtain that the quasiasymptotic behavior (\ref{eqquasi1}) in the space $\mathcal{D}(0,\infty)$ is equivalent to the same quasiasymptotic behavior in the space $\mathcal{D}(\mathbb{R})$, and, because of the structural theorem for quasiasymptotic boundedness \cite[Thm.~2.42, p.~163]{p-s-v} (see also \cite{vindas,vindas3}), we obtain that  $R\in\mathcal{O}'_{C}(\mathbb{R})$ is equivalent to the Ces\`{a}ro behavior (\ref{eqcesaro}) for every $\alpha$.

Note that we have $E(x) \log^{n} x=O(x/\log^{\alpha} x)$ $(\mathrm{C})$ for every $\alpha>0$ as well. So the bounds (\ref{eqextra1}) can be obtained from these Ces\`{a}ro asymptotic estimates by integration by parts. The bound (\ref{eqextra2}) under the assumption (\ref{variationmeasurebd}) can be shown via a similar argument to the one used in the proof of Proposition \ref{conversetauberian}. It is enough to show the bound for $\sigma=\Re e\: s>1$. Consider the splitting
$$
F(s)=\int_{1^{-}}^{e^{X}} x^{-s}\mathrm{d}E(x)+ \int_{e^X}^{\infty} x^{-s}\mathrm{d}E(x),
$$
with $X\gg1$. We can actually assume that $1<\sigma<2$ and $|t|\gg1 $ because otherwise $F$ is already bounded in view of (\ref{variationmeasurebd}). The first term in this formula is clearly $O(\log X)$ because of (\ref{variationmeasurebd}). We handle the second term via integration by parts. Let $E_{m}$ be an $m$-primitive of $E(x)/x$ such that $E_{m}(x)=O(x^{m}/\log ^{2}x)$. The absolute value of the term in question is then
$$
\leq |s|\cdots |s+m| \left(C+ \left|\int_{e^{X}}^{\infty}\frac{E_{m}(x)}{x^{s+m}}\mathrm{d}x\right|\right)\leq C_{m} |t|^{m+1}  \int_{e^{X}}^{\infty}\frac{\mathrm{d}x}{x\log^{2} x}=C_{m}\frac{\ |t|^{m+1}}{X},
$$
and we obtain $F(s)=O(\log |t|)$ by taking $X=|t|^{m+1}$
\end{proof}

With the same technique as the one employed in Lemma \ref{leminversezeta}, one shows the following bound on the inverse of $\zeta$:

\begin{lemma} \label{leminversezeta2} Suppose that condition $(iii)$ from Theorem \ref{thPNTRCesaro} is satisfied. Then, $(s-1)\zeta(s)$ has no zeros on $\Re e\:s\geq 1$ and, in particular, $1/\zeta(s)$ has a $C^{\infty}$-extension to $\Re e\:s\geq 1$ as well. Furthermore, there is some $\beta > 0$ such that 
$$
\frac{1}{\zeta(s)} = O((1+|s|)^{\beta}), \quad \mbox{on } \Re e\:s\geq 1.
$$
\end{lemma}

Let us point out that the Ces\`{a}ro asymptotics  (\ref{asymNNremainder2}) always leads to $N(x)\sim ax$, while (\ref{relationPi2}) leads to $\Pi(x)\sim x/\log x$, which can be shown by standard Tauberian arguments. This comment allows us the application of Theorem \ref{lemsufcondcesaro} to the functions $E_{1}(x)=N(x)-ax$ and $E_2(x)=\Pi(x)-\operatorname*{Li} (x)$.

The rest of the proof goes exactly along the same lines as that of Theorem \ref{thPNTNyman} (using Theorem \ref{lemsufcondcesaro} instead of Theorem \ref{thtaub1} and Proposition \ref{conversetauberian}), and we thus omit the repetition of details. So, Theorem \ref{thPNTRCesaro} has been established.

\end{document}